\documentclass[12pt]{amsart} 
\usepackage{amsmath,amsthm,amsfonts,amssymb,latexsym}
\begin{document}

\theoremstyle{plain}

\newtheorem{thm}{Theorem}[section]
\newtheorem{lem}[thm]{Lemma}
\newtheorem{pro}[thm]{Proposition}
\newtheorem{cor}[thm]{Corollary}
\newtheorem{que}[thm]{Question}
\newtheorem{rem}[thm]{Remark}
\newtheorem{defi}[thm]{Definition}
\newtheorem{Question}[thm]{Question}
\newtheorem{Conjecture}[thm]{Conjecture}

\newtheorem*{thmA}{Theorem A}
\newtheorem*{thmB}{Theorem B}
\newtheorem*{thmC}{Theorem C}

\newtheorem*{thmAcl}{Main Theorem$^{*}$}
\newtheorem*{thmBcl}{Theorem B$^{*}$}

\newcommand{\Maxn}{\operatorname{Max_{\textbf{N}}}}
\newcommand{\Syl}{\operatorname{Syl}}
\newcommand{\dl}{\operatorname{dl}}
\newcommand{\Con}{\operatorname{Con}}
\newcommand{\cl}{\operatorname{cl}}
\newcommand{\Stab}{\operatorname{Stab}}
\newcommand{\Aut}{\operatorname{Aut}}
\newcommand{\Ker}{\operatorname{Ker}}
\newcommand{\fl}{\operatorname{fl}}
\newcommand{\Irr}{\operatorname{Irr}}
\newcommand{\IBr}{\operatorname{IBr}}
\newcommand{\SL}{\operatorname{SL}}
\newcommand{\GL}{\operatorname{GL}}
\newcommand{\FF}{\mathbb{F}}
\newcommand{\NN}{\mathbb{N}}
\newcommand{\N}{\mathbf{N}}
\newcommand{\C}{\mathbf{C}}
\newcommand{\OO}{\mathbf{O}}
\newcommand{\F}{\mathbf{F}}

\renewcommand{\labelenumi}{\upshape (\roman{enumi})}

\newcommand{\PSL}{\operatorname{PSL}}
\newcommand{\PSU}{\operatorname{PSU}}

\providecommand{\V}{\mathrm{V}}
\providecommand{\E}{\mathrm{E}}
\providecommand{\ir}{\mathrm{Irr_{rv}}}
\providecommand{\Irrr}{\mathrm{Irr_{rv}}}
\providecommand{\re}{\mathrm{Re}}

\def\Z{{\mathbb Z}}
\def\C{{\mathbb C}}
\def\Q{{\mathbb Q}}
\def\irr#1{{\rm Irr}(#1)}
\def\ibr#1{{\rm IBr}(#1)}
\def\irrv#1{{\rm Irr}_{\rm rv}(#1)}
\def \c#1{{\cal #1}}
\def\cent#1#2{{\bf C}_{#1}(#2)}
\def\syl#1#2{{\rm Syl}_#1(#2)}
\def\nor{\triangleleft\,}
\def\oh#1#2{{\bf O}_{#1}(#2)}
\def\Oh#1#2{{\bf O}^{#1}(#2)}
\def\zent#1{{\bf Z}(#1)}
\def\det#1{{\rm det}(#1)}
\def\ker#1{{\rm ker}(#1)}
\def\norm#1#2{{\bf N}_{#1}(#2)}
\def\alt#1{{\rm Alt}(#1)}
\def\iitem#1{\goodbreak\par\noindent{\bf #1}}
 \def \mod#1{\, {\rm mod} \, #1 \, }
\def\sbs{\subseteq}

\def\gc{{\bf GC}}
\def\ngc{{non-{\bf GC}}}
\def\ngcs{{non-{\bf GC}$^*$}}
\newcommand{\notd}{{\!\not{|}}}
\newcommand{\Out}{{\mathrm {Out}}}
\newcommand{\Mult}{{\mathrm {Mult}}}
\newcommand{\Inn}{{\mathrm {Inn}}}
\newcommand{\IBR}{{\mathrm {IBr}}}
\newcommand{\IBRL}{{\mathrm {IBr}}_{\ell}}
\newcommand{\IBRP}{{\mathrm {IBr}}_{p}}
\newcommand{\ord}{{\mathrm {ord}}}
\def\id{\mathop{\mathrm{ id}}\nolimits}
\renewcommand{\Im}{{\mathrm {Im}}}
\newcommand{\Ind}{{\mathrm {Ind}}}
\newcommand{\diag}{{\mathrm {diag}}}
\newcommand{\soc}{{\mathrm {soc}}}
\newcommand{\End}{{\mathrm {End}}}
\newcommand{\sol}{{\mathrm {sol}}}
\newcommand{\Hom}{{\mathrm {Hom}}}
\newcommand{\Mor}{{\mathrm {Mor}}}
\newcommand{\Mat}{{\mathrm {Mat}}}

\def\rank{\mathop{\mathrm{ rank}}\nolimits}
\newcommand{\Tr}{{\mathrm {Tr}}}
\newcommand{\tr}{{\mathrm {tr}}}
\newcommand{\Gal}{{\it Gal}}
\newcommand{\Spec}{{\mathrm {Spec}}}
\newcommand{\ad}{{\mathrm {ad}}}
\newcommand{\Sym}{{\mathrm {Sym}}}
\newcommand{\Char}{{\mathrm {char}}}
\newcommand{\pr}{{\mathrm {pr}}}
\newcommand{\rad}{{\mathrm {rad}}}
\newcommand{\abel}{{\mathrm {abel}}}
\newcommand{\codim}{{\mathrm {codim}}}
\newcommand{\ind}{{\mathrm {ind}}}
\newcommand{\Res}{{\mathrm {Res}}}
\newcommand{\Ann}{{\mathrm {Ann}}}
\newcommand{\Ext}{{\mathrm {Ext}}}
\newcommand{\Alt}{{\mathrm {Alt}}}
\newcommand{\Br}{{\mathrm {Br}}}
\newcommand{\AAA}{{\sf A}}
\newcommand{\SSS}{{\sf S}}
\newcommand{\CC}{{\mathbb C}}
\newcommand{\CB}{{\mathbf C}}
\newcommand{\RR}{{\mathbb R}}
\newcommand{\QQ}{{\mathbb Q}}
\newcommand{\ZZ}{{\mathbb Z}}

\newcommand{\NB}{{\mathbf N}}
\newcommand{\ZB}{{\mathbf Z}}
\newcommand{\EE}{{\mathbb E}}
\newcommand{\PP}{{\mathbb P}}
\newcommand{\GC}{{\mathcal G}}
\newcommand{\HC}{{\mathcal H}}
\newcommand{\GA}{{\mathfrak G}}
\newcommand{\TC}{{\mathcal T}}
\newcommand{\SC}{{\mathcal S}}
\newcommand{\RC}{{\mathcal R}}
\newcommand{\bG}{{ \bf G}}
\newcommand\bH{{\bf H}}
\newcommand{\bL} {{\bf L}}
\newcommand{\bU} {{\bf U}}
\newcommand{\bM}{{\bf M}}
\newcommand{\bT}{{\bf T}}
\newcommand{\bX}{{\bf X}}
\newcommand{\bY}{{\bf Y}}
\newcommand{\GCD}{\GC^{*}}
\newcommand{\TCD}{\TC^{*}}
\newcommand{\FD}{F^{*}}
\newcommand{\GD}{G^{*}}
\newcommand{\HD}{H^{*}}
\newcommand{\GCF}{\GC^{F}}
\newcommand{\TCF}{\TC^{F}}
\newcommand{\PCF}{\PC^{F}}
\newcommand{\GCDF}{(\GC^{*})^{F^{*}}}
\newcommand{\RGTT}{R^{\GC}_{\TC}(\theta)}
\newcommand{\RGTA}{R^{\GC}_{\TC}(1)}
\newcommand{\Om}{\Omega}
\newcommand{\eps}{\epsilon}
\newcommand{\al}{\alpha}
\newcommand{\chis}{\chi_{s}}
\newcommand{\sigmad}{\sigma^{*}}
\newcommand{\PA}{\boldsymbol{\alpha}}
\newcommand{\gam}{\gamma}
\newcommand{\lam}{\lambda}
\newcommand{\la}{\langle}
\newcommand{\ra}{\rangle}
\newcommand{\hs}{\hat{s}}
\newcommand{\htt}{\hat{t}}
\newcommand{\tn}{\hspace{0.5mm}^{t}\hspace*{-0.2mm}}
\newcommand{\ta}{\hspace{0.5mm}^{2}\hspace*{-0.2mm}}
\newcommand{\tb}{\hspace{0.5mm}^{3}\hspace*{-0.2mm}}
\def\skipa{\vspace{-1.5mm} & \vspace{-1.5mm} & \vspace{-1.5mm}\\}
\newcommand{\tw}[1]{{}^#1\!}
\renewcommand{\mod}{\bmod \,}



\title[Blocks with transitive fusion systems]
{Blocks with transitive fusion systems}

\author[H\'ethelyi]{L\'aszl\'o H\'ethelyi}
\address{Department of Algebra, Budapest University of Technology and Economics, H-1521 Budapest, Hungary}
\email{hethelyi@math.bme.hu}

\author[Kessar]{Radha Kessar}
\address{Department of Mathematics, City University London, Northampton Square, London, EC1V 0HB, Great Britain}
\email{radha.kessar.1@city.ac.uk}

\author[K\"ulshammer]{Burkhard K\"ulshammer}
\address{Institut f\"ur Mathematik, Friedrich-Schiller-Universit\"at, 07743 Jena, Germany}
\email{kuelshammer@uni-jena.de}

\author[Sambale]{Benjamin Sambale}
\address{Institut f\"ur Mathematik, Friedrich-Schiller-Universit\"at, 07743 Jena, Germany}
\email{benjamin.sambale@uni-jena.de}

\thanks{The first author gratefully acknowledges support by the Hungarian National Science Foundation Research No.\ 77476.
The third author gratefully acknowledges support by the DFG (SPP 1388). The fourth author gratefully acknowledges support by the Carl Zeiss Foundation and the Daimler and Benz
Foundation.}



\begin{abstract}
Suppose that all nontrivial subsections of a $p$-block $B$ are conjugate (where $p$ is a prime). By using the classification of the finite simple groups, we prove that the defect groups 
of $B$ are either extraspecial of order $p^3$ with $p \in \{3,5\}$ or elementary abelian.
\end{abstract} 

\maketitle

\section{Introduction}
Let $p$ be a prime, and let ${\mathcal F}$ be a saturated fusion system on a finite $p$-group $P$ (cf. \cite{AKO} and \cite{CravenBook}). We call ${\mathcal F}$ \textit{transitive} if any two nontrivial elements in $P$ are ${\mathcal F}$-conjugate. 
In this case, $P$ has exponent $\exp(P) \leq p$, and $\Aut_{\mathcal F}(P)$ acts transitively on $\mathrm{Z}(P) \setminus \{1\}$. This paper is motivated by the following:

\begin{Conjecture}
(cf. \cite{KNST}) Let ${\mathcal F}$ be a transitive fusion system on a finite $p$-group $P$ where $p$ is a prime. Then $P$ is either extraspecial of order $p^3$ or elementary 
abelian.
\end{Conjecture}

Moreover, if $P$ is extraspecial of order $p^3$ then results by Ruiz and Viruel \cite{RV} imply that $p \in \{3,5,7\}$. 
Note that the conjecture is trivially true for $p=2$ since groups of exponent 2 are abelian. Thus Conjecture 1.1 is only of interest for $p>2$.
The aim of this paper is to prove the conjecture above for saturated fusion systems coming from blocks.

\begin{thm}\label{mainthm}
Let $p$ be a prime, and let $B$ be a $p$-block of a finite group $G$ with defect group $P$. 
If the fusion system $\mathcal{F} = \mathcal{F}_P(B)$ of $B$ on $P$ is transitive then $P$ is either extraspecial of order $p^3$ or elementary abelian.
\end{thm}

If $P$ is extraspecial of order $p^3$ then the results in \cite{RV} and \cite{KS} imply that $p \in \{3,5\}$. We call a block $B$ with defect group $P$ and transitive fusion
system $\mathcal{F}_P(B)$ \textit{fusion-transitive}. 
Whenever $B$ has full defect then the theorem is a consequence of the results in \cite{KNST}. In our proof of the theorem above, we will make use of the classification of the finite simple groups.

\section{Saturated fusion systems}

We begin with some results on arbitrary saturated fusion systems.

\begin{pro}\label{indexp}
Let $p$ be a prime, and let $\mathcal{F}$ be a transitive fusion system on a finite $p$-group $P$ where $|P| \geq p^4$. Suppose that $P$ contains an abelian subgroup of index $p$. Then $P$
is abelian.
\end{pro}

\begin{proof}
We assume the contrary. Then $p > 2$.

Suppose first that $P$ contains two distinct abelian subgroups $A,B$ of index $p$. Then $AB = P$, $A \cap B \subseteq \mathrm{Z}(P)$ and $|P: A \cap B| = p^2$. Since $P$ is nonabelian we
conclude that $|P:\mathrm{Z}(P)| = p^2$. Thus $1 \neq P' \subseteq \mathrm{Z}(P)$. Since $\mathrm{Aut}_\mathcal{F}(P)$ acts transitively on $\mathrm{Z}(P) \setminus \{1\}$, we conclude that
$P' = \mathrm{Z}(P)$. Hence there are $x,y \in P$ such that $P = \langle x,y \rangle$. Then $P' = \langle [x,y]\rangle$ (cf. III.1.11 in \cite{H1}); in particular, we have $|P'| = p$ and $|P| = p^3$, a contradiction.

It remains to consider the case where $P$ contains a unique abelian subgroup $A$ of index $p$. 
Let $Z$ be a subgroup of order $p$ in $\mathrm{Z}(P)$, and let $B$ be an arbitrary subgroup of order $p$ 
in $A$. By transitivity, there is an isomorphism $\phi: B \longrightarrow Z$ in $\mathcal{F}$. By definition, $Z$ is fully $\mathcal{F}$-normalised. Thus, by Proposition 4.20 in 
\cite{CravenBook}, $Z$ is also fully $\mathcal{F}$-automised and receptive. Hence $\phi$ extends to a morphism $\psi: N_\phi \longrightarrow P$ in $\mathcal{F}$. Since $|B| = p$ we have
$$A \subseteq \mathrm{N}_P(B) = \mathrm{C}_P(B) \subseteq N_\phi$$
(cf. p. 99 in \cite{CravenBook}). Since $\psi(A)$ is also an abelian subgroup of index $p$ in $A$ we conclude that $\psi(A) = A$. Thus $\psi|A \in \mathrm{Aut}_\mathcal{F}(A)$, and $\psi|A$
maps $B$ to $Z$. This shows that $\mathrm{Aut}_\mathcal{F}(A)$ acts transitively on the set of subgroups of order $p$ in $A$. 

In the following, we view $A$ as a vector space over $\mathbb{F}_p$ and $G := \mathrm{Aut}_\mathcal{F}(A)$ as a subgroup of $\mathrm{GL}(A)$. If $S$ denotes the group of scalar matrices in
$\mathrm{GL}(A)$ then $H := GS$ is a transitive subgroup of $\mathrm{GL}(A)$. The transitive linear groups were classified by Hering (cf. \cite{H} or Remark XII.7.5 in \cite{HB3}). We are
going to use the list in Theorem 15.1 of \cite{Habil}. 

Before we do this, we observe the following. By the uniqueness of $A$, $A$ is fully $\mathcal{F}$-automised, i.e. $P/A = \mathrm{N}_P(A)/\mathrm{C}_P(A) \in \mathrm{Syl}_p(\mathrm{Aut
}_\mathcal{F}(A))$. Thus $G = \mathrm{Aut}_\mathcal{F}(A)$ and $H = GS$ both have a Sylow $p$-subgroup of order $p$.

Now we write $|A| = p^n$ and go through the list in Theorem 15.1 of \cite{Habil}:

(i) $H \subseteq \Gamma\mathrm{L}_1(p^n)$; in particular, $|H|$ divides $|\Gamma\mathrm{L}_1(p^n)| = n(p^n-1)$. \\ 
In this case we can identify $A$ with the finite field $L := \mathbb{F}_{p^n}$. Moreover, $P$ is the semidirect product of $L$ with $B = \langle \beta \rangle$ where $\beta$ is a field
automorphism of $L$. For $x \in L$, we have $x\beta \in P$ and
$$1 = (x\beta)^p = x\beta x\beta \ldots x\beta = x\beta(x)\beta^2(x) \ldots \beta^{p-1}(x) = \mathrm{N}_K^L(x)$$
where $K$ is the fixed field of $\beta$. However, it is known that $\mathrm{N}_K^L(L) = K$, a contradiction.

(ii) $n=km$ where $k \ge 2$ and $\mathrm{SL}_k(p^m) \trianglelefteq H$. \\
Since the Sylow $p$-subgroups of $H$ have order $p$, we conclude that $m=1$ and $k=2$. Then $n=2$ and $|P| = p^3$, a contradiction.

(iii) $n=km$ where $k \ge 4$ is even and $\mathrm{Sp}_k(p^m)' \trianglelefteq H$. \\
Since $p > 2$ we have $\mathrm{Sp}_k(p^m)' = \mathrm{Sp}_k(p^m)$. Thus $\mathrm{Sp}_k(p^m)$ has a Sylow $p$-subgroup of order $p^{k^2/4} \geq p^4$, a contradiction.

(iv) $n = 6m$, $p=2$ and $G_2(2^m)' \trianglelefteq H$. \\
This case is impossible as $p>2$.

(v) $n=2$ and $p \in \{5,7,11,19,23,29,59\}$. \\
Then $|P| = p^3$ which is again a contradiction.

(vi) $n=4$, $p=2$ and $H \cong \mathfrak{A}_7$. \\
This case is also impossible as $p>2$.

(vii) $n=4$, $p=3$ and $H$ is one of the groups in Table 15.1 of \cite{Habil}. \\
In this case we have $|P| = 3^5 = 243$. Then Proposition 15.12 in \cite{Habil} leads to a contradiction.

(viii) $n=6$, $p=3$ and $H \cong \mathrm{SL}_2(13)$. \\
In this case we have $|P| = 3^7 = 2187$. However, one can check that $P$ has exponent 9 in this case, a contradiction.
\end{proof}

\begin{pro}\label{indec}
Let $P$ be a nonabelian $p$-group with a transitive fusion system. Then $P$ is indecomposable (as a direct product).
\end{pro}

\begin{proof}
Let $P = N_1 \times \cdots \times N_k$ be a decomposition into indecomposable factors $N_i\ne 1$. Assume by way of contradiction that $k \geq 2$. Since $P$ carries a transitive fusion system we
have
$$\mathrm{Z}(N_1) \times \cdots \times \mathrm{Z}(N_k) = \mathrm{Z}(P) \subseteq P' = N_1' \times \cdots \times N_k'.$$
Let $1 \neq x \in \mathrm{Z}(N_1)$. By hypothesis there exists $\alpha \in \mathrm{Aut}(P)$ such that $\alpha(x) \in \mathrm{Z}(P) \setminus (\mathrm{Z}(N_1) \cup \ldots \cup \mathrm{Z}
(N_k))$. By the Krull-Remak-Schmidt Theorem (see Satz I.12.5 in \cite{H1}) there is a normal automorphism $\beta$ of $P$ such that $\beta(N_i) = \alpha(N_1)$ for some $i \in \{1,\ldots,k\}$.
In particular, there is $y \in \mathrm{Z}(N_i)$ such that $\beta(y) = \alpha(x)$. By Hilfssatz I.10.3 in \cite{H1}, for every $g \in P$ there is a $z_g \in \mathrm{Z}(P)$ such that 
$\beta(g) = gz_g$. Obviously the map $P \longrightarrow \mathrm{Z}(P)$, $g \longmapsto z_g$, is a homomorphism. Since $\mathrm{Z}(N_i) \subseteq N_i'$, we obtain $z_y = 1$. This gives
the contradiction $\alpha(x) = \beta(y) = y \in \mathrm{Z}(N_i)$. 
\end{proof}


\begin{pro}\label{wreathprod}
Let $P = \prod_{i=1}^\infty P_i^{a_i}$ where $P_i = C_{p^{r_i}} \wr C_p \wr \ldots \wr C_p$ ($i$ factors in the wreath product) and $a_i \in \mathbb{N}_0$, $r_i \in \mathbb{N}$ for $i \in \mathbb{N}$. 
Moreover, let $U$ be a normal subgroup of $P$ such that $P/U$ is cyclic, and let $Z$ be a cyclic subgroup of $\mathrm{Z}(U)$. Suppose that $R := U/Z$ supports a transitive fusion system. Then $R$ has order $p^3$ or is
elementary abelian.
\end{pro}

\begin{proof}
We assume the contrary. Then $|R| \geq p^4$ and $p>2$.

Suppose first that $r_j>1$ for some $j>1$. Since $p>2$, $P'$ contains a subgroup isomorphic to $C_{p^{r_j}} \times C_{p^{r_j}}$. Since $P' \subseteq U$ we conclude that $\exp(R) \geq p^2$, a
contradiction. 

Thus $r_j = 1$ for $j>1$, and $P_j$ is the iterated wreath product of $j$ copies of $C_p$ in this case.

Suppose next that $a_j > 0$ for some $j \geq 3$. Since $p>2$, $P'$ contains a subgroup isomorphic to $P_{j-1} \times P_{j-1}$. 
By Satz III.15.3 in \cite{H1}, $P_{j-1}$ has exponent $p^{j-1} \geq p^2$. Since $P' \subseteq U$ we conclude that $\exp(R) \geq p^2$, a contradiction again.

Thus $P = P_1^{a_1} \times P_2^{a_2}$ where $P_1 = C_{p^{r_1}}$ and $P_2 = C_p \wr C_p$. If $a_2 \leq 1$ then $P$ and $R$ contain abelian subgroups of index $p$. In this case Proposition 2.2
gives a contradiction. 

Hence we may assume that $a_2 \geq 2$. Let $\pi: P \longrightarrow P_2^{a_2}$ be the relevant projection. Since $\exp(P_2) = p^2$ we cannot have $\pi(U) = P_2^{a_2}$. On the other hand, $P_2/P_2'$
is elementary abelian. Since $P_2^{a_2}/\pi(U)$ is cyclic, $\pi(U)$ is a maximal subgroup of $P_2^{a_2}$. Let $\pi_1: P_2^{a_2} \longrightarrow P_2^{a_2-1}$ be the projection onto the direct
product of the first $a_2-1$ copies of $P_2$, and let $\pi_2: P_2^{a_2} \longrightarrow P_2^{a_2-1}$ be the projection onto the direct product of the last $a_2-1$ copies of $P_2$. 

Now suppose that $a_2 \geq 3$. Then an argument similar to the one above shows that $\pi_1(\pi(U))$ is a maximal subgroup of $P_2^{a_2-1} = \pi_1(P_2^{a_2})$. Thus $\mathrm{Ker}(\pi_1) \subseteq
\pi(U)$ and, similarly, $\mathrm{Ker}(\pi_2) \subseteq \pi(U)$. Thus $\pi(U)$ contains a subgroup isomorphic to $P_2^2$. Hence $\exp(R) \geq p^2$, a contradiction. 

We are left with the case $a_2 = 2$, i.e. $P = A \times P_2 \times P_2$ where $A = P_1^{a_1} \cong C_{p^{r_1}}^{a_1}$ is abelian. Since $\pi(U)$ is a maximal subgroup of $P_2 \times P_2$, we see that
$A \times \pi(U)$ is a maximal subgroup of $P$. Let $x \in P$ such that $P = U \langle x \rangle$. Then $U \langle x^p \rangle \subseteq A \times \pi(U)$. Since $|P:U\langle x \rangle| \leq p$ we
conclude that $U\langle x^p \rangle = A \times \pi(U)$. Note that $x^p \in \mho(P) \subseteq \mathrm{Z}(P)$.

Suppose that $\exp(A) > p$, and choose an element $a \in A$ of maximal order. We write $x = x_1x_2$ with $x_1 \in A$ and $x_2 \in P_2^2$, we write $a = ux^{pi}$ with $u \in U$ and $i \in 
\mathbb{Z}$, and we write $u = u_1u_2$ with $u_1 \in A$ and $u_2 \in P_2^2$. Then $a^p = u^px^{p^2i} = u_1^p x_1^{p^2i} u_2^px_2^{p^2i} = u_1^px_1^{p^2i} u_2^p$. We conclude that $u_2^p = 1$
and $a^p = u_1^p x_1^{p^2i}$. Thus $p < \exp(A) = |\langle a \rangle | = |\langle u_1 \rangle | = |\langle u \rangle|$, and $1 \neq u^p \in \mho(U) \cap A$. 

By Aufgabe III.15.36 in \cite{H1}, the elements of order $1$ or $p$ form a union of two maximal subgroups. Thus $P_2^2$ contains $p^{2p-2}(2p-1)^2 < p^{2p+1}$ elements of order $1$ or $p$. 
Hence $\pi(U)$ contains elements of order $p^2$; in particular, $\mho(U)$ is noncyclic. Since $\mho(U) \subseteq Z$, this is a contradiction. 

This contradiction shows that $\exp(A) \leq p$, i.e. $P = A \times P_2 \times P_2$ where $A$ is elementary abelian. Hence $P/P'$ is elementary abelian. Since $P/U$ is cyclic we conclude that $U$
is a maximal subgroup of $P$. Thus $U = A \times \pi(U)$ and $\mho(U) \subseteq \pi(U)$. Since $\pi(U)$ contains elements of order $p^2$, we have $1 \neq \mho(U) \subseteq Z$. On the other hand,
Satz III.15.4 in \cite{H1} implies that $\mathrm{Z}(U)$ is elementary abelian. Thus $|Z| = p$ and $Z = \mho(U) \subseteq \pi(U)$. Since $R$ supports a transitive fusion system we have
$$AZ/Z \subseteq \mathrm{Z}(U)/Z \subseteq \mathrm{Z}(R) \subseteq R' = U'Z/Z = \pi(U)'Z/Z \subseteq \pi(U)/Z.$$
Therefore $A=1$, i.e. $P = P_2 \times P_2$. Recall that $U$ is a maximal subgroup of $P$ and that $\pi_1, \pi_2: P \longrightarrow P_2$ denote the two projections. Without loss of generality we
have $\pi_1(U) = P_2$. Since $\mho(U) $ is cyclic, $K_1 := \mathrm{Ker}(\pi_1)$ has order $p^p$ and exponent $p$. 

If $\pi_2(U) \neq P_2$ then $U = P_2 \times \pi_2(U)$ and $\exp(\pi_2(U)) = p$. Thus $Z = \mho(U) \subseteq P_2 \times 1$ and $R \cong P_2/Z \times \pi_2(U)$, a contradiction to Proposition 2.2.

Thus we must also have $\pi_2(U) = P_2$. Then also $K_2 := U \cap \mathrm{Ker}(\pi_2)$ has order $p^p$ and exponent $p$. Moreover, we have $K_1 \times K_2 \subseteq U$.

We may choose elements $x,y \in U$ such that $\pi_1(x)$ and $\pi_2(x)$ have order $p^2$. Since $\langle x^p \rangle = Z = \langle y^p \rangle$ we see that $\pi_2(x)$ and $\pi_1(y)$ have order
$p^2$. However, we may choose $y$ such that $yK_1$ contains an element $y'$ such that $\pi_2(y')$ has order $p$. Since $\pi_1(y) = \pi_1(y')$ still has order $p^2$, we have a final contradiction.
\end{proof}

\section{Blocks}

We now present the proof of Theorem~\ref{mainthm}.

\begin{proof} 
Suppose that the result is false. Then $P$ is nonabelian with $|P| \geq p^4$ and $p>2$.


By \cite[Proposition IV.6.3]{AKO} we may assume that $B$ is quasiprimitive. This means that, for any normal subgroup $H$ of $G$, $B$ covers a unique $p$-block of $H$.

Now let $H$ be a normal subgroup of $G$, and let $b$ be the unique $p$-block of $H$ covered by $B$. Suppose that $P \cap H = 1$. (This is satisfied, for example, whenever $H$ is a $p'$-subgroup.) Then $b$ has defect 
zero. By Clifford theory, there exist a finite group $G^\ast$, a central $p'$-subgroup $H^\ast$ of $G^\ast$, and a $p$-block $B^\ast$ of $G^\ast$ with defect group $P^\ast \cong P$ such that $\mathcal{F}_{P^\ast}(B^\ast)$ is
equivalent to $\mathcal{F}$. Thus we may replace $G$ by $G^\ast$ and $B$ by $B^\ast$. 

Repeating the argument above we may therefore assume that every normal subgroup $H$ of $G$ with $P \cap H = 1$ is central. In particular, we have $\mathrm{O}_{p'}(G) \subseteq \mathrm{Z}(G)$. 

It is well-known that $M := \mathrm{O}_p(G) \subseteq P$. Suppose first that $M \neq 1$. Since $\mathcal{F}$ is transitive this implies $M = P$. Then $\Phi(P)$ is a normal subgroup of $G$ and properly contained in $P$. Since $\mathcal{F}$ is
transitive, we must have $\Phi(P) = 1$. Thus $P$ is elementary abelian in this case.

Hence, in the following, we may assume that $\mathrm{O}_p(G) = 1$. Then $\mathrm{F}(G) = \mathrm{O}_{p'}(G) = \mathrm{Z}(G)$. 
Moreover, the layer $\mathrm{E}(G)$ is nontrivial. Let $b$ be the unique $p$-block of $\mathrm{E}(G)$ covered by $B$. 
Then $b$ has defect group $P \cap \mathrm{E}(G) \neq 1$. Since $B$ is transitive, this implies that $P \subseteq \mathrm{E}(G)$.

Let $L_1,\ldots, L_n$ denote the components of $G$. Then $\mathrm{E}(G) = L_1 \ast \cdots \ast L_n$ is a central product. For
$i=1,\ldots,n$, the unique $p$-block $b_i$ of $L_i$ covered by $b$ has defect group $P_i := P \cap L_i \neq 1$. Moreover, we
have $P = P_1 \times \cdots \times P_n$. Since $\mathcal{F}$ is transitive, this implies that $n = 1$. Thus $\mathrm{E}(G) = L_1 =: L$
is quasisimple, and $G/\mathrm{Z}(G)$ is isomorphic to a subgroup of $\Aut(L)$. 


If $|P| = p^4$ then Proposition 15.14 in \cite{Habil} gives a contradiction. Thus we may assume that $|P| \geq p^5$; in particular,
$|L|$ is divisible by $p^5$. If $P$ is a Sylow $p$-subgroup of $G$ then the results of \cite{KNST} imply our theorem. 
Hence we may assume that $|G|$ is divisible by $p^6$.

We now make use of the classification of the finite simple groups and discuss the various possibilities for the simple group
$\mathrm{F}^\ast(G)/\mathrm{Z}(G) \cong L/\mathrm{Z}(L)$. Since $\mathcal{F}$ is transitive we have $\mathrm{C}_L(u) \cong \mathrm{C}_L(v)$ for
any $u,v \in P \setminus \{1\}$. This will be a very useful fact.


It can be checked with GAP \cite{GAP} that $L/\mathrm{Z}(L)$ cannot be a sporadic simple group. Similarly, $L/\mathrm{Z}(L)$ cannot be a simple group with an exceptional Schur multiplier.

Suppose that $L = \mathfrak{A}_n$ is an alternating group. Then $P$ is a defect group of a $p$-block of $\mathfrak{A}_n$. Hence $P$ is
also a defect group of a $p$-block of the symmetric group $\mathfrak{S}_n$. Thus $P$ is a direct product of (iterated) wreath
products of groups of order $p$. Since $C_p \wr C_p$ has exponent $p^2$ we conclude that $P$ is a direct product of groups of
order $p$, and the result follows in this case.

Suppose next that $L = \hat{\mathfrak{A}}_n$ is the $2$-fold cover of $\mathfrak{A}_n$. We may assume that $b$ is a faithful block of $\hat{\mathfrak{A}}_n$. In this case the defect groups
of $b$ have a similar structure as those in $\mathfrak{A}_n$ (cf. \cite[Theorem~5.8.8]{NaTs}), so we are done here by the same argument. 

Suppose now that $L/\mathrm{Z}(L)$ is a group of Lie type in characteristic $p$. Then the $p$-block $b$ of $L$ has full defect, i.e.
$P$ is a Sylow $p$-subgroup of $L$. Since $\mathcal{F}$ is transitive, every nontrivial element $u \in P$ is conjugate in $G$ to an element
$v \in \mathrm{Z}(P)$. Thus $|L:\mathrm{C}_L(u)| = |L:\mathrm{C}_L(v)|$ is not divisible by $p$. Therefore the results in \cite{NT}
imply that $P$ is abelian.

Finally suppose that $L/\mathrm{Z}(L)$ is a group of Lie type in characteristic $r \neq p$. 
First we deal with the exceptional groups of Lie type. Let $S\in\Syl_p(L)$. 
By \S 10.1 in \cite{GL}, $S$ contains an abelian normal subgroup $N$ such that $S/N$ is isomorphic to a subgroup of the Weyl group of $L/\mathrm{Z}(L)$. 
If $|S/N|\le p$, then Proposition~\ref{indexp} gives a contradiction. 
This already implies the claim for $p\ge 7$. Now let $p=5$. Then by the same argument we may assume that $L/\mathrm{Z}(L)\cong E_8(q)$ where $q\equiv\pm1\pmod{5}$. This case will be handled in Section~\ref{p5}.
Now let $p=3$. Here we need to discuss the following groups: $F_4$, $E_6$, $^2E_6$, $E_7$ and $E_8$. 
For $L/\mathrm{Z}(L)\cong F_4(q)$ we have $|P|\le p^6$ and the result follows by Proposition~15.13 in \cite{Habil}. 
The remaining cases will be discussed in Section~\ref{p5}.

We may therefore assume that $L/\mathrm{Z}(L)$ is a classical group. In this case our theorem follows from the results of the next section.
\end{proof}

\section{Classical Groups in non-describing characteristic}

We keep the notation of the previous section.
We suppose in this section that $L/\mathrm{Z}(L)$ is a simple group of Lie type in characteristic $r$, $r\ne p $. 
Let $q$ be a power of $r$. Suppose that $ L= \bL^F/ Z$, 
where $\bL$ is a simple simply connected algebraic group defined over an algebraic closure $\bar {\mathbb F}_q$ of a field $\mathbb{F}_q$ of $q$ elements, 
$ F: \bL \to \bL $ a Frobenius morphism with respect to an ${\mathbb F}_q$-structure on $\bL $ and $ Z $ is a central subgroup of $\bL^F $. Note that by the classification of finite simple groups, we may assume that if $q$ is a power of $2$, then $\bL $ is not of type $C_n$.
Let $\tilde b$ be the block of $\bL^F$ dominating $b$ and $\tilde P$ be a defect group of $\tilde b$ such that $\tilde P Z/Z= P$.

We define groups $\bH$ as follows. If $L/\mathrm{Z}(L) = B_n(q)$, then $\bH= \mathrm{SO}_{2n+1}( \bar{\mathbb F}_q) $. If $L/\mathrm{Z}(L)= C_n(q) $, then $\bH=\mathrm{Sp}_{2n}(\bar {\mathbb F}_q)$. 
If $L/\mathrm{Z}(L)= D_n^{\pm}(q)$, then 
$\bH= \mathrm{SO}_{2n}(\bar{\mathbb F}_q )$. Here, if $q$ is a power of $2$, and $\bL $ is of type $B_n$, then by $\mathrm{SO}_{2n}(\bar{\mathbb F}_q )$ we mean the adjoint simple group of type $B_n$. If $q$ is a power of $2$ and if $\bL $ is of type $D_n $, then by $\mathrm{SO}_{2n}(\bar{\mathbb F}_q )$ we mean the simple algebraic group of type $D_n$ corresponding to the root datum $(X, \Phi, Y, \Phi^{\vee} ) $ for which the fundamental roots are $ e_1-e_2, e_2-e_3, \ldots, e_{n-1}-e_{n}, e_{n-1}+e_{n} $ and $X = \{ \sum_{i=1}^n a_ie_i \, : a_i \in {\mathbb Z} \} $ for an orthonormal basis, $ e_1, e_2, \cdots, e_n $, of $n$-dimensional Euclidean space. We may and will assume that $\bH$ is an $F$-stable quotient of $\bL$.

\begin{pro} \label{classical} Suppose that $p$ is an odd prime and $L/\mathrm{Z}(L)$ is a classical group in non-describing characteristic different from triality $D_4$. Suppose that $B$ is a fusion-transitive block with $P$ of order at least $p^5 $. Then $P$ is abelian.
\end{pro}

\begin{proof} Suppose that $ L/\mathrm{Z}(L)$ is the projective special linear group $\mathrm{PSL}_n(q)$, so $\bL =\mathrm{SL}_n(\overline{\mathbb F}_q)$ and $ L = \mathrm{SL}_n(q)$.
Let $D$ be a defect group of a block of $\mathrm{GL}_n(q)$ covering $\tilde b $ such that 
$\tilde P = D \cap \mathrm{SL}_n(q)$. By the results of Fong and Srinivasan on blocks of finite general linear groups \cite[Theorem~(3C)]{FS}, $D$ is isomorphic to the Sylow $p$-subgroup of a direct product of general linear groups over finite extensions of ${\mathbb F}_q$. Since $\mathrm{Z}(L)$ and $D/\tilde P$ are cyclic, the claim follows from Proposition~\ref{wreathprod}. The case that $ L/\mathrm{Z}(L)$ is the projective special unitary group can be handled similarly.

Now consider the case that $L/\mathrm{Z}(L)$ is of type $B$, $C$ or $D$. Then $\tilde P$ is a defect group of $ \bL^F$. Let $1\ne z \in \mathrm{Z}(\tilde P)$. Since $p$ is odd, $ \mathrm{C}_{\bL} (z)$ is a Levi subgroup of $\bL$.
For any subset $A$ of $\bL$, denote by $\overline {A } $ the image of $A $ under the isogeny from $\bL $ onto $\bH$ and denote by $U$ the kernel of the isogeny. Since $U$ is a central $2$-subgroup of $\bL $, $\overline {\mathrm{C}_{\bL}( z)} = \mathrm{C}_{\bH}(\bar z) $. 

The group $ \mathrm{C}_{\bH}(\bar z) $ is a direct product 
$$ \mathrm{C}_{\bH}(\bar z) = \bH_0 \times \cdots \times \bH_r, $$ where $ \bH_0 $ is either the identity or a classical group and for $ i\geq 1 $, $ \bH_i $ is a direct product of general linear groups with $F$ transitively permuting the factors. This follows easily from the standard description of the root datum of $\bH$. So,
$$ \mathrm{C}_{\bH}(\bar z) ^F= \bH_0^F \times \cdots \times \bH_r^F, $$ where $\bH_i^F $ is a finite general linear or unitary group for $i\geq 1 $ and $\bH_0^F $ is a finite classical group (possibly the identity). 

Let $\bL_i $ be the inverse image in $\mathrm{C}_{\bL} (z) $ of $\bH_i $, $ 0\leq i \leq r $. Then $\bL_i $ is a normal $F$-stable subgroup of $ \mathrm{C}_{\bL}(z) $, $ \mathrm{C}_{\bL}(z) = \bL_0 \cdots \bL_r $ and 
$$[\bL_i, \bL_0\cdots \bL_{i-1}\bL_{i+1} \cdots \bL_r] \leq \bL_i \cap ( \bL_0\cdots \bL_{i-1} \bL_{i+1} \cdots \bL_r) = U. $$ 

We claim that $\overline { \bL_i ^F } $ is a normal subgroup of $\bH_i ^F$ of $2$-power index. Indeed, let $M$ be the inverse image in $ \bL_i $ of $\bH_i^F $. Then $M$ is $F$-stable since $U$ is $F$-stable. Further, $[M, F] \leq U $. Since $U$ is central in $ M$, the map $ M \to U $ defined by $ x \to x^{-1}F(x) $ is a group homomorphism. The kernel of this map is $\bL_i^F$ whence $\bL_i^F$ is a normal subgroup of $M$ and the index of $\bL_i^F$ in $M$ divides $|U|$. The claim follows since $U$ is a $2$-group. 

The claim implies that $ \bL_0^F \cdots \bL_r ^F$ is a normal subgroup of $2$-power index of
 $\mathrm{C}_{\bL}(z) ^F$. So, $\tilde P $ is a defect group of $ \bL_0^F \cdots \bL_r ^F$. The commutator relationship given above then implies that 
$\tilde P $ is a direct product $P_0 \cdots P_r $, where $P_i $ is a defect group of $ \bL_i^F $, $ 0\leq i \leq r $. By Proposition~\ref{indec}, $\tilde P= P_i $ for some $i$, $1 \leq i \leq r $. Since $z $ is central in $\mathrm{C}_{\bL} (z) $, $i \geq 1 $ and $\bH_i ^F $ is a general linear or unitary group with a central $p$-element. 
Let $ R= \tilde P \cap [\bL_i, \bL_i]^F$, a defect group of $[ \bL_i, \bL_i]^F$. By suitably replacing $\tilde P $ by an $ \bL_i ^F$-conjugate, we may assume that the relevant block of $[ \bL_i, \bL_i]^F $ is $\tilde P$-stable and hence that $\tilde P$ is a defect group of $ [ \bL_i, \bL_i]^F \tilde P $.

The isogeny $ \bL_i \to \bH_i $ restricts to an isogeny 
$[\bL_i, \bL_i] \to [\bH_i, \bH_i]$ with kernel $ U \cap [\bL_i, \bL_i ]$. However $[\bH_i, \bH_i]$ is a simply connected semisimple group, being the direct product of special linear groups. Thus, $ U \cap [\bL_i, \bL_i ] =1$ and the restriction of the isogeny to $[\bL_i, \bL_i ]$ is an abstract group isomorphism from $[\bL_i, \bL_i]$ to $[\bH_i, \bH_i]$ which commutes with $F$. Consequently, $[\bL_i, \bL_i] ^F \cong [\bH_i,\bH_i ]^F$. Also, $ U\cap [\bL_i, \bL_i ]\tilde P =1$ and 
the induced map $ [\bL_i, \bL_i]^F \tilde P \to \bH_i^F $ is injective. Thus $\overline {\tilde P} \cong \tilde P \cong P $ is a defect group of 
$ \overline{ [\bL_i, \bL_i]^F \tilde P} \cong [\bH_i, \bH_i]^F \overline {\tilde P} $. Since $\bH_i^F $ is a finite general linear or unitary group, the result now follows from \cite[Theorem~(3C)]{FS} and Proposition~\ref{wreathprod} in the same way as for the case that $L/\mathrm{Z}(L) $ is a projective special linear or unitary group. 
\end{proof}

%

\section{On $A_{p-1}$-components}

\begin{lem}\label{A4exp} Suppose that $p$ is an odd prime and let $G$ be a finite group isomorphic to one of the groups $\mathrm{SL}_p(q)$ or $\mathrm{SU}_p(q)$ 
for some prime power $q$ not divisible by $p$.
Let $U$ be a non-abelian $p$-subgroup of $G $. Then $U$ contains a normal abelian subgroup $U_0$ of index $p$ such that 
any element of $ U \setminus U_0 $ has order $p$.
If $ |U| \geq p^{p+1}$, then $ U_0$ contains an element of order $p^2 $. 
\end{lem}

\begin{proof} First, consider the case that $G$ is special linear or unitary. By replacing $q$ if necessary by some power we may assume that 
$ U\leq \mathrm{SL}_p(q) $ and $ p $ divides $q-1 $.
Let $ S_0$ be the Sylow $p$-subgroup of the group of diagonal matrices of $\mathrm{SL}_p (q)$ and let 
$\sigma $ be a non-diagonal, monomial matrix in $\mathrm{SL}_p(q)$ of order $p$. Then $ S:=\langle S_0, \sigma\rangle $ is a Sylow $p$-subgroup of $\mathrm{SL}_p(q)$, $S_0$ is normal in $S$, abelian, of index $p$ in $S$, rank $p-1$ and any element of $ S$ not in $S_0$ has order $p$. Let $ U_0 = U\cap S_0 $. Then $U_0 $ has index at most $p$ in $U$. On the other hand, since $U$ is non-abelian and $S_0 $ is abelian, $U $ is not contained in $U_0$. Thus $U_0$ has index $p$ in $U$, proving the first assertion. Now suppose that $U$ has exponent $p$. Then $ U_0 $ is elementary abelian. On the other hand, $ U_0\leq S_0 $ and the $p$-rank of $S_0 $ is $p-1$. Hence, $ |U | =p|U_0 | \leq p^p $.

\end{proof}

In the rest of this section, $p$ will denote a fixed prime and $\bG$ will denote a connected reductive group in characteristic $r \ne p $ with a Frobenius morphism $F$ with respect to some ${\mathbb F }_{r'} $ structure for some power $r'$ of $r$. In what follows, whenever we talk of a component of $\bG$ , we will mean a simple component of $[\bG, \bG]$.

We need a slight variation of the previous lemma.

\begin{lem}\label{A4indexp} Suppose that $p$ is odd. If $[\bG, \bG] = \SL_p$, then any $p$-subgroup of $ \bG^F $ has an abelian subgroup of index $p$. 
\end{lem} 

\begin{proof} Since $ \bG = {\mathrm Z}^{\circ}(\bG) [\bG, \bG] $ any element and hence any subgroup of $\bG^F $ is contained in $\mathrm{Z}^{\circ} (\bG)^{F^d} [\bG, \bG]^{F^d} $ for some $ d\geq 1$. This can be seen as follows. Since $\bG = \mathrm{Z}^{\circ}(\bG)[\bG, \bG]$, any element $u$ of $\bG $ can be written in the form $u= xy$, where $x \in \mathrm{Z}^{\circ}(\bG)$ and $y \in [\bG, \bG]$. Let $\iota: \bG \to \GL_n$ be an embedding. Then for some power, say $F^t$ of $F$, some power, say $s$ of $r$, and for all $g\in \bG $, $ F^t \circ\iota (g) = F_s (\iota (g) )$ where $F_s $ is the standard Frobenius morphism of $\GL_n $ raising every matrix entry to the $s$-th power. The claim follows since for any $ h\in \GL_n $, $F_s^m (h)= h$ for some natural number $m$. Since any Sylow $p$-subgroup of $\mathrm{Z}^{\circ} (\bG)^{F^d} [\bG, \bG]^{F^d} $ is of the form $R_1 R_2 $, where $R_1 $ is a Sylow $p$-subgroup of $\mathrm{Z}^{\circ} (\bG)^{F^d} $ and $R_2 $ is a Sylow $p$-subgroup of $[\bG, \bG]^{F^d} $, the result follows from the previous Lemma and the fact that $R_1 $ is central in $R_1 R_2 $. 
\end{proof}

\begin{lem}\label{A2-punch} Suppose that $p$ is odd. Let $\bX=\SL_p$ be an $F$-stable component of $\bG $ such that $\bX^F$ has a central element of order $p$ and let $ \bY$ be the product of all other components of $\bG$ and $\mathrm{Z}^{\circ} (\bG )$. 
Let $ P$ be a $p$-subgroup of $\bG^F$ such that $ P \cap \bX^F$ is non-abelian of order at least $p^p$ and $P$ is not contained in $\bX^F \bY^F $. Then there exists an element of order $p^2 $ in $P$. Further, if $Z$ is a central subgroup of $\bG^F$ of order $p$ such that $ P/Z $ has exponent $p$, then $ Z \leq \bX^F$. 
\end{lem} 

\begin{proof} Let $ \tilde P $ be the inverse image of $ P$ under the surjective group homomorphism $\bX \times \bY \to \bG $ induced by multiplication. 
The kernel of the multiplication map is isomorphic to $\bX \cap \bY = \mathrm{Z}(\bX) \cap \mathrm{Z}(\bY) $. Since $\bX $ is a simple group of type $A_{p-1}$, the kernel of the multiplication map is a group of order $p$ and in particular, $ \tilde P $ is a finite $p$-group. Let $P_1 \leq \bX $ be the image of $\tilde P$ under the projection of $\bX \times \bY \to \bX $. Clearly $P_1 $ contains $P \cap \bX^F $. We claim that $P\cap \bX^F$ is proper in $P_1$. Indeed, otherwise $ \tilde P \leq (P\cap \bX^F ) \times \bY $, whence $P \leq (P \cap \bX^F ) \bY $. This implies that $ P \leq (P \cap \bX^F ) (P\cap \bY^F ) \leq P \cap \bX^F \bY^F$, a contradiction. Since $ P\cap \bX^F$ is assumed to have order at least $p^p$, the claim implies that $|P_1| \geq p^{p+1} $. 

Now $P_1 $ is a finite subgroup of $\bX$, thus of some finite special linear (or unitary) group. Hence, by Lemma ~\ref{A4exp}, there exists an element $ x \in P_1 $ of order $p^2 $. Let 
$ y\in \bY $ be such that $w= xy \in P $. Since $ P\cap \bX^F$ is non-abelian again by Lemma~\ref{A4exp}, there exists $\sigma \in P \cap \bX^F $ such that $ x\sigma $ has order $p$. Then $w $ and $ w \sigma \in P $, $ w^p= x^py^p $ and $ (w\sigma )^p = y^p$. Then either $ w^p\neq 1 $ or $ (w\sigma )^p \ne 1 $, proving the first part of the result. Suppose that $ P/Z$ has exponent $p$. Then, $w^p, (w\sigma)^p$ are in $Z$. Hence $ x^p \in Z $. Since $ 1\ne x^p $ and $ Z $ has order $p$ the second assertion follows.
\end{proof} 

\begin{lem} \label{split-stab} Let ${\mathcal X}$ be an $F$-stable subset of components of $\bG$. Let $\bX$ be the product of all elements of ${\mathcal X} $ and let $ \bY $ be the product of ${\mathrm Z}^{\circ} (\bG)$ and all the components of $[\bG, \bG] $ not in $ {\mathcal X} $.
\begin{enumerate}
\item Let $ P $ be a defect group of a block $b$ of $ \bG^F$. Then $ P \cap \bX^F \bY^F $ is a defect group of a block of $\bX^F \bY^F $ covered by $b$ and is of the form $P_1 P_2$, where $P_1$ is a defect group of a block of $\bX^F $ covered by $b$ and $P_2$ is a defect group of a block of $\bY^F$ covered by $b$.
If $\mathrm {Z} (\bX)^F \cap \mathrm{Z} (\bY)^F $ has $p'$-order, then $ P= P _1 P_2 $ and the product is direct.
\item Let $c$ be a $p$-block of $ \bX^F\bY^F $. Then the index of the stabiliser of $c$ in $\bG^F$ is prime to $p$. 
Suppose further that $\mathrm{Z}(\bX)^F \cap \mathrm{Z}(\bY)^F $ is a $p$-group. Then $c $ is $\bG^F$-stable, $c$ is covered by a unique block of $\bG^F$ and if $P$ is a defect group of the block of $\bG^F$ covering $c$, then $P \cap \bX^F \bY^F $ is a defect group of $c$ and $ P/(P \cap \bX^F \bY^F ) \cong \bG^F/\bX^F\bY^F$. 
\end{enumerate}
\end{lem} 

\begin{proof} The first statement of (i) follows from the theory of covering blocks as $ \bX^F \bY^F $ is a normal subgroup of $\bG^F$, $ \bX^F $ and $ \bY^F $ centralise each other and $ \bX^F \cap \bY^F = \mathrm {Z} (\bX)^F \cap \mathrm{Z} (\bY)^F \subseteq {\mathrm Z} (\bG )^F $ is central in $\bX^F \bY^F$. The second assertion of (i) follows from the first assertion, the fact that $|\bG^F| = |\bX^F| |\bY^F|$ and $ \bX^F \cap \bY^F = \mathrm {Z} (\bX)^F \cap \mathrm{Z} (\bY)^F $.

We now prove (ii). Let $u \in \bG^F$ be a $p$-element. Then $u=xy $, with $x \in \bX $ and $ y\in \bY$ such that $x^{-1}F(x) =yF(y^{-1})$ is an element of $\mathrm{Z}(\bX) \cap \mathrm{Z}(\bY)$. We may assume without loss of generality that $x$ and $y$ are $p$-elements. The block $c$ of $\bX^F \bY^F$ is a product $c_1c_2 $ of blocks $c_1$ of $\bX^F$ and $c_2$ of $\bY^F$. Thus, it suffices to prove that $\,^x c_1 =c_1 $ and $\,^yc_2 =c_2 $.

Now consider a regular embedding $ \bX \leq \tilde \bX $, where $\tilde \bX$ is a connected reductive group with connected centre containing $\bX $ as a closed subgroup, such that $[\tilde \bX, \tilde \bX ] =[\bX, \bX]$ and such that $F$ extends to a Frobenius morphism of $ \tilde \bX$. Since $ x^{-1}F(x) \in \mathrm{Z}(\bX) \leq \mathrm{Z}^{\circ} (\tilde \bX)$, $ x= x_1z $ for some $ x_1 \in \tilde \bX^F$, and $z \in \mathrm{Z}^{\circ} (\tilde \bX)$. We may assume also that $x_1 $ is a $p$-element. Then $\,^xc_1= \,^{x_1}c_1 $. 
On the other hand, $c_1 $ contains an ordinary irreducible character $\chi$ in a Lusztig series corresponding to a semisimple element of order prime to $p$ in the dual group of $\bX $, hence the index in $\tilde \bX ^F$ of the stabiliser in $\tilde \bX^F$ of $\chi$ has order prime to $p$ (see for instance \cite[Corollaire~11.13] {Bon}). 
This proves the first assertion. If $\mathrm{Z}(\bX)^F \cap \mathrm{Z}(\bY)^F $ is a $p$-group, then $|\bG^F/\bX^F \bY^F| = |\mathrm {Z} (\bX)^F \cap \mathrm{Z} (\bY)^F |$ is a power of $p$. By the first assertion, $c$ is $\bG^F$-stable and by standard block theory, there is a unique block of $\bG^F$ covering $c$. The second assertion of (ii) now follows from (i).
\end{proof}

\begin{lem} \label{A2-abelpunch}
Suppose that $p$ is odd. Let $\bX $ be an $F$-stable component of $\bG $ of type $A_{p-1} $ and let $ \bY$ be the product of all other components of $\bG$ and $\mathrm{Z}^{\circ} (\bG )$. Suppose that $\mathrm{Z}(\bX)^F \cap \mathrm{Z}( \bY ) ^F \ne 1 $ and that $ P$ is a defect group of $\bG^F$ such that $ P \cap \bX^F$ is abelian. Then there exists an $F$-stable torus $\bT $ of $\bX $ such that $P$ is a defect group of $ (\bY \bT)^F $. 
\end{lem}

\begin{proof} In the proof, we will identify blocks with the corresponding central primitive idempotents. Let $b$ be a block of $\bG^F$ with $P$ as defect group and let $P_0: =P \cap \bX^F\bY^F $. The hypothesis implies that $|\mathrm{Z}(\bX)^F \cap \mathrm{Z}(\bY)^F|=p $. So, by Lemma~\ref{split-stab}, $b$ is a block of $\bX^F \bY^F $, $P_0$ is a defect group of $b$ as block of $\bX^F \bY^F$ and $ P/P_0 $ is isomorphic to $ \bG^F/\bX^F \bY^F $. Let $ b= b_1 b_2 $, where $ b_1 $ is the block of $\bX^F$ covered by $b$ and $b_2 $ is the block of $\bY^F$ covered by $b$. 

Let $ u\in P$ generate $P$ modulo $P_0$ and write $u =xy $, $x\in \bX $, $y\in \bY $. Since $u$ is a $p$-element, we may assume that both $x$ and $y$ are $p$-elements. 

Now consider an $F$-compatible regular embedding of $\bX$ in $\tilde \bX$ such that $ \tilde \bX ^F$ is a finite general linear (or unitary) group. Since $\mathrm{Z}(\tilde \bX)$ is connected, there exists $z \in \mathrm{Z}^{\circ}(\tilde \bX)$ such that $g:=xz^{-1} \in \tilde \bX ^F$. Further, we may choose $z$ such that $g$ is a $p$-element. Since $u=xy$ normalises $P_1$, $x$ normalises $P_1 $ and therefore $g$ normalises $P_1 $. Therefore 
$S= \langle P_1, g \rangle \leq \tilde X^F $ is a $p$-group. Since $u$ normalises $b_1$ it also follows that $ b_1 $ is $S$-stable.

We claim that there exists a block of $\tilde \bX^F $ covering $b_1 $ with a defect group $D$ containing $S$. Indeed, in order to prove the claim, it suffices to prove that $\Br_S (b_1) \ne 0 $. Since $b_1 $ and $b_2 $ are both $\bG^F$-stable,
$$ 0 \ne \Br_P (b) = \Br_P(b_1) \Br_P(b_2) $$ and consequently $\Br_P(b_1) \ne 0 \ne \Br_P(b_2) $. Hence writing $b_1 =\sum_{v \in \bX^F}\alpha_v v$ as an element of the modular group algebra of $\bX^F$ there exists $v \in \bX^F$ with $\alpha_v $ non-zero such that $v$ centralises $P$ and in particular $v$ centralises $P_1 $ and $u$. Since $z$ is central, and $ y $ centralises $\bX$, we have that $v$ also commutes with $g$. Hence $v$ centralises $S$ and it follows that $\Br_S (b_1) \ne 0 $, proving the claim.

By the block theory of finite general linear (or unitary) groups (see \cite{FS}; noting that $p$ divides $q-1 $ in the linear case and that $p$ divides $q+1 $ in the unitary case) $D$ is a Sylow $p$-subgroup of the centraliser of some semisimple element of $\tilde \bX^F$. Since by hypothesis $P_1 =D\cap \bX^F$ is abelian, we have that $D$ is abelian, hence $D$ is the Sylow $p$-subgroup of $ \tilde \bT^F$ for some $F$-stable maximal torus $\tilde \bT $ of $\tilde \bX $. Set $\bT = \bX \cap \tilde \bT $, an $F$-stable maximal torus of $\bX$. Then $ P_1 = D \cap \bX^F $ is a Sylow $p$-subgroup of $\bT^F $. Now $ g =xz \in S \leq D \leq \tilde \bT $, and $ z \in \tilde \bT $ (as $z $ is central), hence $ x=gz^{-1}\in \tilde \bT \cap \bX = \bT$. 

Set 
$\bG_0=\bT\bY $. We have $u=xy \in \bG_0^F $. Since $ \bX \cap \bY \leq \mathrm{Z}(\bX)\leq \bT $, we have that $ \bG_0^F \cap \bX^F\bY^F =\bT^F\bY^F$ and $\bG_0^F/\bT^F \bY^F $ is isomorphic to a subgroup of $\bG^F /\bX^F\bY^F $ and in particular has order $p$. 
Hence $ \bG_0 ^F = \langle \bT^F\bY^F, u \rangle $. Let $ e $ be a block of $\bT^F$ such that $ e b_2 \ne 0 $. Since $ \bT^F$ and $ \bY^F$ commute, $eb_2 $ is a block of $\bT^F \bY^F$. Since $\bT $ is central in $\bG_0 $, $ e$ is $\bG_0^F$-stable. Further, $b_2$ is $P$-stable hence $ b_2 $ is $\bG_0^F$-stable. So $eb_2$ is a $\bG_0^F$-stable block of $\bT^F \bY^F$ and therefore a block of $\bG_0^F$. 
Since $P_1 $ is the Sylow $p$-subgroup of $\bT^F$ and $\bT^F$ is abelian, $P_1 $ is the defect group of $e$ and $P_2 $ is a defect group of $ b_2 $. Thus, $P_1P_2 $ is a defect group of $ eb_2 $ as block of $\bT^F \bY^F$. Since $ \Br_P (eb_2) =\Br_P(e) \Br_P(b_2) $ is non-zero, it follows by order considerations that $P$ is a defect group of $ eb_2$.
\end{proof}

\section{The case $p=3, 5$}\label{p5} 
In this section we handle the remaining exceptional groups of Lie type for $p\le 5$.
\begin{lem} \label{gen-morita-defect} Let $G$, $H$ be finite groups, $B$ a $p$-block of $G$ and $C$ a $p$-block of $H$ such that $B$ and $C$ are Morita equivalent. Let $P$ be a defect group of $ B$, and $Q$ a defect group of $C$. Suppose that $P$ has exponent $p$. Then $P$ is abelian if and only if $Q$ is abelian. Further, $P$ has an abelian subgroup of index $p$ if and only if $Q$ has an abelian subgroup of index $p$. 
\end{lem} 

\begin{proof} By \cite[Satz~J]{Ku81}, the exponent of defect groups is an invariant of Morita equivalence, hence $Q$ has exponent $p$. In particular any abelian subgroup of $P$ or of $Q$ is elementary abelian. The remaining statements follow by the fact that Morita equivalence preserves the rank of the corresponding defect groups (see \cite[Theorem~2.6]{BW}).
\end{proof}

\begin{lem}\label{general-bonrou} Let $\bL $ be connected reductive, with Frobenius morphism $F$, and let $Z$ be a central $p$-subgroup of $ \bL ^F$. Let $b$ be a block of $\bL^F $ and $P$ a defect group of $b$. Suppose that $ P/Z $ is non-abelian, supports a transitive fusion system and $|P/Z | \geq p^4 $. Let $\bH $ be an $F$-stable Levi subgroup of $\bL$, let $c$ be a Bonnaf\'e-Rouquier correspondent of $b$ in $\bH$ and let $Q$ be a defect group of $c$. Then $Q/Z$ has exponent $p$ and $ Q/Z$ does not have an abelian subgroup of index $p$. In particular, a Sylow $p$-subgroup of $\bH^F $ does not have an abelian subgroup of index $p$.
\end{lem}

\begin{proof} Let $\bar b$ be the block of $\bL^F/Z$ dominated by $b$ and let $\bar c $ be the block of $\bH^F/Z $ dominated by $c$. By 
\cite[Prop.~4.1] {EKKS}, $\bar b $ and $\bar c $ are Morita equivalent. Further, $P/Z$ is a defect group of $\bar b $ and $ Q/Z$ is a defect group of $\bar c $. The result now follows from Lemma~\ref{indexp} and Lemma~\ref{gen-morita-defect}. 
\end{proof}

\begin{pro} \label{E6-redux} Let $\bL $ be connected reductive, in characteristic $r \ne p=3$ with Frobenius morphism $F$, and suppose that 
$ [\bL,\bL] $ is simply connected of type $E_6 $ in characteristic $r \ne 3 $. Let $Z$ be a cyclic subgroup of $ \mathrm{Z}(\bL^F) $ of order $1$ or $3$ and let $P$ be a defect group of $ \bL^F$. Suppose that $P/Z$ supports a transitive fusion system and $ |P/ Z | \geq 3^7$. Suppose further that either $Z=1 $ or that $\bL $ is simple. Then $P/Z$ is abelian.
\end{pro}

\begin{proof} Suppose that $P/Z$ is non-abelian.
Let $\bH $ be an $F$-stable Levi subgroup of $\bL $ and $c$ a block of $\bH^F$ such that $c$ is quasi-isolated and $b$ and $c$ are Bonnaf\'e-Rouquier correspondents. Let $s\in \bH^*$ be a semisimple label of $c $ (and $b $). Since $b$ and $c$ are Bonnaf\'e-Rouquier correspondents, $\operatorname{C}_{\bL^*}(s)= \operatorname{C}_{\bH^*}(s)$. Let $Q$ be a defect group of $c$. By Lemma~\ref{general-bonrou}, we may assume that $ Q/Z$ has exponent $3$ and does not have 
an abelian subgroup of index $3$. Note that all components of $\bL $ and hence of $\bH$ are simply connected.

If $\bH^F$ has a component of type $D_4 $ or $D_5 $, then 
the only other possible components are of type $A_1 $. We get a contradiction by Lemma~\ref{split-stab}(i), Lemma~\ref{general-bonrou} and the fact that finite groups of type $D_4(q) $, $D_5(q) $, $\,^2D_4(q) $, $\, ^2D_5(q) $ and $ \,^3D_4(q) $ have a Sylow $3$-subgroup with an abelian subgroup of index $3$. 

Thus, either all components of $\bH $ are of type $A$ or $\bH$ has a component of type $E_6 $. Let us first consider the case that 
all components of $\bH$ are of type $A$. 
In particular, $\mathrm{C}_{\bH^*}^{\circ} (s) $ is a Levi subgroup of $\bH ^*$ and since $s$ has order prime to $3$, $\mathrm{C}_{\bL^*} (s) = \mathrm{C}_{\bH^*} (s) $ is connected. It follows that $ s$ is central in $ \bH^*$, hence that $Q$ is a defect group of a unipotent block of $\bH^F$.

Suppose that $\bH$ has a component ${\bX} $ of type $A_5 $. Then $\bX $ is $F$-stable and is the only component of $\bH$. If $\bX^F$ does not contain a central element of order $3$, then by Lemma \ref{split-stab}(i), a Sylow $3$-subgroup of $ \bH^F$ is a direct product of a Sylow $3$-subgroup of $\bX^F $ with the Sylow $3$-subgroup of $\mathrm{Z}^{\circ} (\bH)^F $. Furthermore in this case a Sylow $3$-subgroup of ${\bX}^F$ has an abelian subgroup of index $3$. If $\bX^F$ contains a central element of order $3$, then by 
\cite[Prop.~3.3 and Theorem]{CE94}, the principal block is the only unipotent block of $\bX^F$, and it follows that $Q/Z$ has an element of order $9$ since $\mathrm{PSL}_6(q)$ (respectively $\mathrm{PSU}_6(q)$) has elements of order $9$ if $ 3\mid q-1 $ (respectively $ 3 \mid q+1$).

Suppose that $\bH$ has a component of type $A_4 $. Then the only other possible component is of type $A_1 $ and it follows from Lemma~\ref{split-stab}(i) that a Sylow $3$-subgroup of $\bH^F$ has an abelian subgroup of index $3$.

Suppose that $\bH $ has a component $\bX $ of type $A_3 $. If all other components are of type $A_1 $, then the above argument applies. If $\bH$ has a component of type $A_2 $, say $\bY $, then this is the only other component of $ \bH $. If the Sylow $3$-subgroups of $\bX^F$ are abelian, then Lemma~\ref{split-stab}(i) and Lemma~\ref{A4indexp} give the result. Thus, we may assume that the Sylow $3$-subgroups of $\bX^F$ are non-abelian. Thus, $\bX^F$ is isomorphic to $ \mathrm{SL}_4 (q) $ (respectively $\mathrm{SU}_4(q) $) with $ 3 \mid q-1 $ (respectively $3 \mid q+1$). Consequently, the principal block is the unique 
unipotent block of $\bX^F$. In particular, $Q$ contains a Sylow $3$-subgroup of $\bX^F$ and $Q/Z$ has an element of order $ 9$.

Thus, we may assume that all components of $\bH$ are of type $A_2 $ or $A_1 $. 
By rank considerations, there can be at most two components of type $A_2 $. By Lemma~\ref{split-stab} (i) and Lemma~\ref{A4indexp} we may assume that there are two $F$-stable components $\bX $ and $\bY $ of type $A_2 $ such that both $\bX^F $ and $\bY^F$ have central elements of order $3$.
Consequently, the principal block of $\bX^F$ is the only unipotent block of $\bX^F$ and similarly for $\bY^F$. The only other component of $\bH$, if it exists is of type $A_1 $, which also has a unique unipotent block. Hence $ Q $ is a Sylow $3$-subgroup of $ \bH^F$. 

Since $\bH $ is a Levi subgroup of $\bL $, there is surjective group homomorphism from $\mathrm{Z}(\bG)/\mathrm{Z}^{\circ} (\bG) $ to $\mathrm{Z}(\bH)/\mathrm{Z}^{\circ} (\bH)$ (see \cite[Prop.~4.1]{Bon}) and by hypothesis, $[\bL, \bL]$ is simple of type $E_6 $. Hence $\mathrm{Z}(\bH)/ \mathrm{Z}^{\circ} (\bH)$ is cyclic of order $1$ or $3$. Since $\bX$ and $\bY$ are the only components of $\bH$ with central elements of order $3$, it follows that either $\mathrm{Z}(\bX)$ or $\mathrm{Z}(\bY )$ covers 
$\mathrm{Z}(\bH)/ \mathrm{Z}^{\circ} (\bH) $. Thus, either $\mathrm{Z}(\bX) \leq \mathrm{Z}(\bY) \mathrm{Z}^{\circ} (\bH) $ or 
$\mathrm{Z}(\bY) \leq \mathrm{Z}(\bX) \mathrm{Z}^{\circ} (\bH) $. 

Assume that $\mathrm{Z}(\bX) \leq \mathrm{Z}(\bY) \mathrm{Z}^{\circ} (\bH) $. Let $\bU $ be the product of all components of $\bH$ other than $\bX $ and $\mathrm{Z}^{\circ} (\bH)$. Then, $\mathrm{Z}(\bX)^F \leq ( \mathrm{Z}(\bY) \mathrm{Z}^{\circ} (\bH) )^F \leq \bU^F$ and hence $ 3 \mid |\bX^F \cap \bU^F|$. Since $Q$ is a Sylow $3$-subgroup of $\bH^F$ and $|\bH^F|=|\bX^F| |\bU^F|$, $Q$ is not contained in $\bX^F \bU^F$. Further, $Q \cap \bX^F$ is a Sylow $3$-subgroup of $\bX^F$ and in particular is non-abelian of order at least $3^3$. By Lemma~\ref{general-bonrou}, $Q/Z$ has exponent $3$. So, by Lemma~\ref{A2-punch}, 
$1\ne Z \leq \mathrm{Z}(\bX) $ whence $Z=Z(\bX) $. Since $Z\ne 1 $, $ \bL $ is simple by hypothesis. In particular, $Z =\mathrm{Z}(\bX) $ covers $\mathrm{Z}(\bG)/\mathrm{Z}^{\circ}(\bG)$. It follows that $ \mathrm{Z}(\bY ) \leq \mathrm{Z}(\bX) \mathrm{Z}^{\circ} (\bH) $. By the same argument as above with $\bY $ replacing $\bX$, we get that $ Z =\mathrm{Z}(\bY )$. In particular $ \mathrm{Z}(\bX) =\mathrm{Z}(\bY)$, a contradiction since $\bX \cap \bY = 1$.

Finally, consider the case that $ \bH $ has a component of type $E_6 $. Then $\bH= \bL$ and $b=c$. Let $b_0$ be a block of $[\bL, \bL]^F$ covered by $b$ and let $P_0= P \cap [\bL, \bL]^F $ be a defect group of $b_0 $. Let $ R$ be the Sylow $3$-subgroup of $\mathrm{Z}^{\circ} (\bL )^F$. By Lemma~\ref{split-stab}(i) applied with $\bX =[\bL, \bL] $ and $ \bY =\mathrm{Z}^{\circ} (\bL ) $, 
$ P \cap [\bL, \bL]^F \mathrm{Z}^{\circ} (\bL )^F = P_0 R $. So, $P/P_0R $ is a subgroup of $ \bL^F/ ([\bL, \bL]^F \mathrm{Z}^{\circ} (\bL )^F ) $. Since
 $ \bL^F/ ([\bL, \bL]^F \mathrm{Z}^{\circ} (\bL )^F) $ is either trivial or has order $3$, we have that $ P_0R $ has index at most $3$ in $P$. If $P_0$ is abelian, then $P$ and hence $P/Z$ has an abelian subgroup of index $3$. Thus, $P_0 $ is non-abelian. We claim that $R \leq P_0$. Indeed, by hypothesis,
either $Z=1 $ or $[\bL, \bL]=\bL $. If $\bL=[\bL, \bL]$, then $R=1 $ and the claim holds trivially. If $Z=1 $, then $P$ supports a transitive fusion system.
Hence $ R \leq \mathrm{Z}(P) \leq [P, P] \leq [\bL, \bL]^F$ and the claim is proved. Thus, $P_0=PR $ has index at most $3$ in $P$.

Assume first that $b_0$ is unipotent. The unipotent $3$-blocks of exceptional groups have been described in \cite{En00}. 
If $b_0$ is the principal block, then $P/Z $ has exponent greater than $3$. So, $b_0$ is non-principal and $P_0$ is non-abelian. By \cite{En00} (last part of the proofs for Tableau I), $P_0$ is the extension of a homocyclic group, say $T$, of rank $2$ by a group of order $3$. If $T$ is not elementary abelian, then $ TZ/Z $ has exponent at least $9$ and hence so does $ P/Z$. Thus, we may assume that $T$ is elementary abelian. So, $ |P_0| =3^3 $ and $|P| \leq 3^4$, a contradiction.

So, we may assume that $b_0$ is quasi-isolated but not unipotent. Here the blocks are described in \cite[Section 4.3]{KesMa}. In particular, $b_0$ corresponds to one of lines 13, 14, or 15 of Table~4 of \cite{KesMa} (and the corresponding Ennola duals; see the last remark of Section~4 of \cite{KesMa}). If $b_0$ corresponds to line~15, then $P_0$ is abelian. If $b_0$ corresponds to line 14, then $P_0 $ is the extension of a homocyclic group, say $T$, of rank $4$ by a group of order $3$. If $T$ is not elementary abelian, then $TZ/Z $ has exponent at least $9$ and if $T$ is elementary abelian, then $|P_0| \leq 3^5 $, whence $|P| \leq 3^6$, a contradiction. If $b_0$ corresponds to line 13, then $P_0$ contains a subgroup isomorphic to a Sylow $3$-subgroup of $\mathrm{SL}_6(q)$ with $3\mid q-1$. 
In particular, $ \mho^{1}(P) $ is not cyclic. On the other hand, since $P/Z$ has exponent $3$, $ \mho^{1}(P) \leq Z $. This is a contradiction as $Z$ is cyclic.
\end{proof} 

\begin{pro} \label{E7E8} Suppose that either $p=3 $ and $ \bL $ is simple and simply connected of type $E_7 $ or $E_8 $ in characteristic $r \ne 3 $ or that $p=5 $ and $\bL$ is simple of type $E_8$ in characteristic $r \ne 5 $. Let $ F $ be a Frobenius morphism on $\bL $ and let $P$ be a defect group of a $p$-block of $ \bL^F$. Suppose that $P$ supports a transitive fusion system and $ |P | \geq 3^7$ if $p=3 $. Then $P$ is abelian. 
\end{pro}
\begin{proof} Suppose if possible that $P$ is not abelian. As before $ P$ has exponent $p$, and is indecomposable and $P$ does not have an abelian subgroup of index $p$. Let $ z \in \mathrm{Z}(P)$. Since $\bL$ is simply connected, $\bH:= \mathrm{C}_{\bL }(z)$ is a connected reductive subgroup of $\bL$ of maximal rank and of semisimple rank at most $8$ and by \cite[Chapter 5,~Theorem~9.6]{NaTs}, $P$ is a defect group of $\bH^F$. The possible components of $\bH$ are of type $A$, $D$, $E_6 $ or $E_7$. 

Let $ {\mathcal X} $ be an $F$-stable subset of components of $\bH $ and let $ \bX$ be the product of the elements of ${\mathcal X} $. Suppose that $ \bX^F$ does not have a central element of order $p$. By Lemma~\ref{split-stab}(i), $ P = (P\cap \bX^F ) \times (P \cap \bY^F) $ where $\bY $ 
is the product of $\mathrm{Z}^{\circ} (\bH )$ and all components of $\bH$ other than those in ${\mathcal X}$. The indecomposability of $ P$ implies that either $ P \leq \bX^F $ or $ P\leq \bY^F $. Since $z$ is a central $p$-element of $\bH^F $, and $\bX^F$ does not have a central element of order $p$, it follows that $ P \leq \bY^F $. By replacing $\bH$ by $\bY $, we may assume that the fixed points of every $F$-orbit of components of $\bH $ have central elements of order $p$ ($\bY $ may have rank less than $\bH $). Thus, if $p=5 $ the only possible components are of type $A_4 $ and if $p=3 $, then the only possible components are of type $A_2 $, $A_5$, $A_8$ or $E_6 $.

Suppose that $\bH$ has an $F$- stable component $\bX $ of type $A_{p-1} $. Let $\bY $ be the product of all components of $\bH$ other than those in $\bX $ with $\mathrm{Z} ^{\circ}(\bH )$. By Lemma~\ref{split-stab}(i) and the indecomposability of $P$, we may assume that $\mathrm {Z} (\bX)^F \cap \mathrm {Z} (\bY)^F $ and hence $\bH^F/\bX^F\bY^F $ has order $p$. So, by Lemma~\ref{split-stab}(ii), $P $ is not contained in $\bX^F \bY^F$.
By Lemma~\ref{A2-abelpunch}, we may assume that $P\cap \bX^F$ is not abelian since otherwise we can replace $\bX $ by a torus. 
Since $\bX^F$ has a central element of order $p$, $\bX ^F $ is a special linear 
(respectively unitary) group. The only non-abelian defect groups of a finite special linear (or unitary) group of degree $p$ in non-describing characteristic are Sylow $p$-subgroups and $ P\cap \bX^F$ is a non-abelian defect group of $\bX^F$. Thus, $P\cap \bX^F$ is a Sylow $p$-subgroup of $\bX^F$ and consequently has order at least $p^p$. Since we have shown above that $P $ is not contained in $ \bX^F \bY^F $, by 
Lemma~\ref{A2-punch}, $ P$ has an element of order $p^2 $, a contradiction.
Thus, we may assume that any component of $ \bH$ of type $A_{p-1}$ lies in an $F$-orbit of size at least $2$.

If $p=5 $, the only case left to consider is that $ \bH$ has two components of type $A_4 $ (and these are the only ones) transitively permuted by $F$. In this case, by rank considerations, $ {\mathrm Z}^{\circ} (\bH)$ is trivial, and hence $\bH^F $ is isomorphic to a special linear or unitary group. In particular the Sylow $5$-subgroups of $\bH^F$ have an abelian subgroup of index $5$, a contradiction. This completes the proof for the case that $p=5 $.
 
Now assume that $p=3 $. Let us first consider the case that there is a component $\bX$ of $\bH$ of type $A_8$. Then $\bH=\bX =\SL_8$ and we may argue as in the first part of the proof of Proposition~\ref{classical}.

Let us next consider the case that there is a component $\bX$ of $\bH$ of type $A_5 $. If $\bX $ also has a component of type $A_2$, then by rank consideration this is the unique component of type $A_2 $ and we have ruled out this situation above. Thus $\bX$ is the unique component of $\bH$. Let $P_0 $ be a defect group of a covered block of $\bX^F$. 
The Sylow $3$-subgroup of $\mathrm{Z}^{\circ}(\bH)^F$ is contained in $\mathrm{Z}(P)$ and $\mathrm{Z}(P) \leq [P,P] \leq [\bX, \bX] \cap \bH^F \leq \bX^F $, hence we have that the Sylow $3$-subgroup of $\mathrm{Z}^{\circ}(\bH) ^F$ is contained in $\bX^F $ and in particular has order at most $3$. Thus, $P_0 $ has index at most $3$ in $P$. 
In particular $ P_0$ is non-abelian.
Now $\bX =\bM/Z$, where $\bM $ is a special linear group of degree $6$ (with a compatible $F$-action) and $Z$ is a central subgroup. Since $\mathrm{Z}(\bM) $ is cyclic of order $6$ (or $3$ if $r=2 $) and since $\bX$ has a central element of order $3$, $ Z$ is either trivial or of order $2$, $Z$ is $F$-stable and $ Z^F=Z $. Further, $\bM^F/Z $ is a normal subgroup of $\bX^F = (\bM/Z)^F$ of index $|Z|$. Thus $P_0$ is a defect group of 
$\bM^F/Z$ and up to isomorphism a defect group of $\bM^F$ and $\bM^F=\mathrm{SL}_6(q) $ (respectively $\mathrm {SU}_6(q)$). Since $\bM^F/Z$ has index prime to $3$, $ \bM^F/Z$ contains the $3$-part of the centre of $\bX^F$, hence $\bM^F$ has a central element of order $3$. Thus, $P_0 $ is the intersection with $\bX^F $ of a Sylow $3$-subgroup of the centraliser of a semisimple $3'$-element of $ \mathrm{GL}_6(q) $ (or $\mathrm{GU}_6(q)$). Since $P_0$ has exponent $3$ and is non-abelian, the possible structures of semisimple centralisers in $\mathrm{GL}_6(q) $ (or $\mathrm{GU}_6(q)$) force that the centraliser in $\mathrm{GL}_6(q)$ (respectively $\mathrm{GU}_6(q)$) has the form $\mathrm{GL}_3 (q^2) $. Hence $|P_0| \leq p^3 $ and $ |P| \leq p^ 4 $ a contradiction. 

Suppose $\bH$ has a component of type $E_6$. Arguing as in the previous case $\bH$ has no components of type $A_2$ and hence the $E_6$-component is the unique component of $\bH$. This component is of simply connected type since as explained in the beginning of the proof we may assume that the $F$-fixed point subgroup of every $F$-orbit of components of $\bH $ has central elements of order $3$ and we are done by Proposition~\ref{E6-redux} (note that we apply Proposition~\ref{E6-redux} here in the case that $Z=1$).

The only case left to consider is that all components of $\bH$ are of type $A_2 $ and no component is $F$-stable. By rank considerations and the fact that groups of type $E_8$ do not have semisimple centralisers with component type $A_2^4 $ (see the tables in \cite{De}), we are left with two possibilities: either $\bH $ has exactly three components, all of type $A_2 $ and in a single $F$-orbit or $\bH $ has exactly two components both of type $A_2 $ and in a single $F$-orbit. 
In any case, $ [\bH, \bH]^F $ has a quotient or subgroup $H_0$ isomorphic to $ \mathrm{PSL}_3 (q) $ (respectively $\mathrm{PSU}_3 (q)$) for some $q$ such that $ |[\bH, \bH]^F |/|H_0| $ equals $1$ or $3$. Let $ P_0 = P \cap [\bH, \bH ]$ and let $P_0'$ be either the intersection of $P_0$ with $H_0$ or the image of $P_0$ in $H_0$. Then $P_0'$ has exponent $3$. Since any $3$-subgroup of a finite projective special linear or unitary group of degree $3$ has an abelian subgroup of index $3$ and since the $3$-rank of these groups is $2$, it follows that $|P_0'| \leq 3^3 $. Hence $|P_0| \leq 3^4 $.

We claim that the index of $P_0$ in $P$ is at most $3$. Indeed,
let $ R$ be the Sylow $3$-subgroup of $\mathrm{Z}^{\circ}(\bH )^F $. Then $ R \leq \mathrm{Z}(P) \leq [P, P] \leq [\bH, \bH] $, that is $ R \leq P_0$. On the other hand, $ |P/P_0 R|$ divides $|\mathrm{Z}([\bH, \bH]^F)|_3 $ and we have seen from the structure of $ [\bH, \bH]^F $ that $ \mathrm{Z}([\bH, \bH]^F )$
has order at most $3$. This proves the claim. Hence $|P| \leq 3^5 $, a contradiction.
\end{proof}

\section{Consequences}

We note some consequences of Theorem~\ref{mainthm}.

\begin{thm}\label{kl1}
Let $B$ be a block of a finite group such that $k(B)-l(B)=1$ (e.\,g. a block with multiplicity $1$). Then $B$ has elementary abelian defect groups.
\end{thm}
\begin{proof}
See proof of Theorem~3.6 in \cite{KNST}.
\end{proof}

\begin{cor}
Let $B$ be a block of a finite group such that $k(B)=3$. Then $B$ has elementary abelian defect groups.
\end{cor}
\begin{proof}
We have $l(B)\in\{1,2\}$. In case $l(B)=1$ it was shown by K\"ulshammer~\cite{Kk4} that the defect groups of $B$ have order $3$. The remaining case $l(B)=2$ follows from Theorem~\ref{kl1}.
\end{proof}


\begin{thebibliography}{10}

\bibitem{AKO}
M. Aschbacher, R. Kessar and B. Oliver, \textit{Fusion systems in algebra and
 topology}, London Mathematical Society Lecture Note Series, Vol. 391,
 Cambridge University Press, Cambridge, 2011.

\bibitem{BW}
C.~Bessenrodt and W.~Willems, 
\textit{Relations between complexity and modular invariants and consequences for $p$-soluble groups},
J. Algebra \textbf{86} (1984), 445--456.

\bibitem{Bon} C. Bonnaf{\'e}, {\em Sur les caract\`eres des groupes r{\'e}ductifs finis a centre non connexe: applications aux groupes sp{\'e}ciaux lin{\'e}aires et unitaires}, Ast{\'e}risque {\bf306} (2006). 

\bibitem{BonRou} 
C. Bonnaf{\'e}, R. Rouquier, 
{\em Cat\'egories d\'eriv\'ees et vari\'et\'es de {D}eligne-{L}usztig }, 
Publ. Math. Inst. Hautes \'Etudes Sci. {\bf 97} (2003), 1--59.

\bibitem {CE94} 
M. Cabanes and M. Enguehard, 
\textit{On unipotent blocks and their ordinary characters}, 
Invent. Math. \textbf{117} (1994), 149--164. 

\bibitem {CE} 
M. Cabanes and M. Enguehard, 
\textit{On blocks of finite reductive groups and twisted induction}, 
Advances Math. \textbf{145} (1999), 189--229. 

\bibitem{Atlas}
J. H. Conway, R. T. Curtis, S. P. Norton, R. A. Parker, and R. A. Wilson,
{\it ATLAS of Finite Groups}, 
Clarendon Press, Oxford, $1985$.

\bibitem{CravenBook}
D.~A. Craven, 
\textit{The Theory of Fusion Systems}, 
Cambridge Studies in Advanced Mathematics, Vol. 131, 
Cambridge University Press, Cambridge, 2011.

\bibitem{De} D.~I.~Deriziotis, \textit{Centralisers of semisimple elements of the Chevalley groups $E_7 $ and $E_8$}. Tokyo. J. Math {\bf 6} (1983), 191--216.

\bibitem{EKKS} C.~Eaton, R.~Kessar, B.~K\"ulshammer, B.~Sambale, \textit{2-blocks with abelian defect groups} Adv. Math. {\bf 254} (2014), 706--735.

\bibitem{En00} 
M. Enguehard, 
\textit{Sur les {$l$}-blocs unipotents des groupes r\'eductifs finis quand {$l$} est mauvais}, 
J. Algebra {\bf 230} (2000), 334--377.

\bibitem{FS} 
P.~Fong and B.~Srinivasan, 
\textit{The blocks of finite general linear and unitary groups}, 
Inv. Math. \textbf{69} (1982), 101--153.


\bibitem{GAP}
The GAP Group,
\textit{GAP -- Groups, Algorithms, and Programming, Version 4.6.4},
2013, {\tt http://www.gap.system.org}.

\bibitem{GL}
D.~Gorenstein and R.~Lyons, 
\textit{The local structure of finite groups of characteristic 2 type},
Mem. Amer. Math. Soc. \textbf{42} (1983).

\bibitem{GLS3} 
D.~Gorenstein, R. Lyons and R. Solomon,
\textit{The classification of finite simple groups 3}, 
Mathematical Surveys and Monographs \textbf{40}, 
American Mathematical Society (1998).

\bibitem{H}
C.~Hering,
\textit{Transitive linear groups and linear groups which contain irreducible subgroups of prime order},
Geom. Dedic. \textbf{2} (1974), 425--460.


\bibitem{H1}
B.~Huppert,
\textit{Endliche Gruppen I},
Springer-Verlag, Berlin 1967.

\bibitem{HB3}
B.~Huppert and N.~ Blackburn,
\textit{Finite groups III},
Springer-Verlag, Berlin 1982.

\bibitem{KesMa} 
R.~Kessar and G.~Malle, 
\textit{Quasi-isolated blocks and {B}rauer's height zero conjecture},
{Ann. of Math. (2)} \textbf{178} (2013) 321--384. 

\bibitem{KS}
R. Kessar and R. Stancu, 
\textit{A reduction theorem for fusion systems of blocks}, 
J. Algebra \textbf{319} (2008), 806--823.

\bibitem{Ku81} 
B.~K\"ulshammer, 
\textit{Bemerkungen \"uber die Gruppenalgebra als symmetrische Algebra II},
{J. Algebra} \textbf{75} (1982), 59-69.

\bibitem{Kk4}
B. K\"ulshammer, \textit{Symmetric local algebras and small blocks of finite
 groups}, J. Algebra \textbf{88} (1984), 190--195.

\bibitem{KNST}
B.~K\"ulshammer, G.~Navarro, B.~Sambale and P.~H.~Tiep, 
\textit{Finite groups with two conjugacy classes of $p$-elements and related questions for $p$-blocks}, 
Bull. London Math. Soc. \textbf{46} (2014), 305--314

\bibitem{NaTs}
H. Nagao and Y. Tsushima, \textit{Representations of finite groups}, Academic
 Press Inc., Boston, MA, 1989.

\bibitem{NT}
G. Navarro and P.~H. Tiep, 
\textit{Abelian {S}ylow subgroups in a finite group},
J. Algebra \textbf{398} (2014), 519-526.

\bibitem{RV}
A. Ruiz and A. Viruel, 
\textit{The classification of {$p$}-local finite groups over the extraspecial group of order {$p^3$} and exponent {$p$}}, 
Math. Z. \textbf{248} (2004), 45--65.

\bibitem{Habil}
B.~Sambale, 
\textit{Blocks of finite groups and their invariants},
Lecture Notes in Math. Vol. 2127, Springer-Verlag, Berlin, 2015.

\bibitem{St} 
R.~Steinberg, 
\textit{Lectures on Chevalley Groups}, 
Notes by J. Faulkner and R. Wilson, Mimeographed notes, Yale University Mathematics Department (1968).

\end{thebibliography}
\end{document}